\documentclass{amsart}
\usepackage{amsmath,amssymb}
\usepackage{amsfonts}
\usepackage{bm}
\usepackage{ascmac}
\usepackage{url}
\newtheorem{thm}{Theorem}[section]
\newtheorem{lem}[thm]{Lemma}
\newtheorem{prop}[thm]{Proposition}
\newtheorem{cor}[thm]{Corollary}

\theoremstyle{definition}
\newtheorem{defn}[thm]{Definition}
\newtheorem{exam}[thm]{Example}
\newcommand{\id}{\mathrm{id}}
\newcommand{\g}{\Gamma}
\newcommand{\cross}[1]{C(#1) {\rtimes}_r \Gamma}
\newcommand{\Cst}{\mathrm{C}^{\ast}}
\newcommand{\Cstr}{\mathrm{C}_r^{\ast}}
\newcommand{\cS}{{\mathcal S}}
\newcommand{\cH}{{\mathcal H}}
\newcommand{\cM}{{\mathcal M}}
\newcommand{\cU}{{\mathcal U}}
\newcommand{\cV}{{\mathcal V}}
\newcommand{\cA}{{\mathcal A}}
\newcommand{\IB}{{\mathbb B}}
\newcommand{\IC}{{\mathbb C}}
\newcommand{\bX}{\tilde{X}}
\title{Uniformly recurrent subgroups and the ideal structure of reduced crossed products}
\author{Takuya Kawabe}
\begin{document}
\maketitle
\markright{UNIFORMLY RECURRENT SUBGROUPS AND THE IDEAL STRUCTURE}
\begin{abstract}
We study the ideal structure of reduced crossed product of topological dynamical systems of a countable discrete group.
More concretely, for a compact Hausdorff space $X$ with an action of a countable discrete group $\Gamma$, we consider the absence of a non-zero ideals in the reduced crossed product $C(X) \rtimes_r \Gamma$ which has a zero intersection with $C(X)$.
We characterize this condition by a property for amenable subgroups of the stabilizer subgroups of $X$ in terms of the Chabauty space of $\Gamma$.
This generalizes Kennedy's algebraic characterization of the simplicity for a reduced group $\mathrm{C}^{*}$-algebra of a countable discrete group.
\end{abstract}
\section{Introduction}
Throughout this paper, $\g$ denotes a countable discrete group. 
We say X is a \emph{compact $\g$-space} if $X$ is a compact Hausdorff space 
with a continuous $\g$-action ${\g \times X  \to  X, \ (t, x) \mapsto tx}$.
We study the ideal structure of  the reduced crossed product $\cross{X}$.
The simplest situation is the following.
\begin{defn}
Let $X$ be a compact $\g$-space. 
We say $C(X)$ \emph{separates the ideals} in $\cross{X}$ if for every ideal $I$ in $\cross{X}$,
we have $I = ( I \cap C(X) ) \rtimes_r \g$.
\end{defn} 
In other words, there is one-to-one correspondence between the ideals in $\cross{X}$ and
the $\g$-invariant ideals in $C(X)$ (see \cite[Propostion 1.1]{sierakowski}). 
\begin{defn}
We say that a compact $\g$-space $X$ satisfies the 
\emph{intersection \ property}
if every non-zero ideal in $\cross{X}$ has a non-zero intersection with $C(X)$.
\end{defn}
Then we have the following result.
\begin{thm}[Sierakowski, {\cite[Theorem 1.10]{sierakowski}}]
Let $X$ be a compact $\g$-space. Then $C(X)$ separates the ideals in $\cross{X}$ if and only if $X$ satisfies the following properties.
\begin{enumerate} 
\item The action of $\g$ on $X$ is exact. 
\item Every $\g$-invariant closed set in $X$ has the intersection property.
\end{enumerate}
\end{thm}
The purpose of this paper is to characterize the intersection property of $\g$-spaces 
in terms of dynamical systems.
For an amenable group $\g$, Kawamura and Tomiyama showed that the intersection properties of compact $\g$-spaces is equivalent to topological freeness.
\begin{thm}[Kawamura--Tomiyama, {\cite[Theorem 4.1]{kt}}]\label{thmkt}
If $\g$ is amenable, the following are equivalent.
\begin{enumerate}
\item The space $X$ has the intersection property.
\item For every $t \in \g \setminus \{e\}$, we have $\mathrm{Fix}(t)^{\circ} = \emptyset$. 
\end{enumerate}
\end{thm}

We say that $\g$ is \emph{$\Cst$-simple} if its reduced group $\Cst$-algebra $\Cstr \g$ is simple.
In recent work \cite{kk}, Kalantar and Kennedy established a dynamical characterization of $\Cst$-simplicity, and Breuillard, Kalantar, Kenndey and Ozawa proved that many groups are $\Cst$-simple. 
In more recent work \cite{kennedy}, Kennedy showed an algebraic characterization of $\Cst$-simplicity, as follows. 
\begin{thm}[Kennedy, {\cite[Theorem 6.3]{kennedy}}] \label{thmken}
A countable discrete group is $\Cst$-simple if and only if it satisfies the following condition:
For every amenable subgroup $\Lambda \le \g$, there exists a sequence $(g_n)$ such that for every subsequence $(g_{n_k})$ of $(g_n)$, we have
\[
\bigcap_k g_{n_k} \Lambda g_{n_k}^{-1} = \{e\}.
\] 
Equivalently, the sequence $(g_{n} \Lambda g_{n}^{-1})$ converges to $\{e\}$ in the Chabauty topology.
\end{thm}
The set $\mathrm{Sub}(\g)$ of all subgroups of $\g$ admits a natural topology, called \emph{Chabauty topology}. We treat $\mathrm{Sub}(\g)$ as a compact $\g$-space with this topology and the $\g$-action by conjugation (see Definition \ref{defcha}).

The first main result of this paper is the characterization of the intersection property by a property for stabilizer subgroups,
which is motivated by the above results Theorem \ref{thmkt} and \ref{thmken}.
\begin{thm}\label{thm1}
Let $X$ be a compact $\g$-space. The following are equivalent.
\begin{enumerate}
    \item Every $\g$-invariant closed set in $X$ has the intersection property.
    \item For every point $x$ in $X$ and every amenable subgroup $\Lambda$ in $\g_x$,
            there is a net $(g_i)$ in $\g$ such that $(g_i x)$ converges to $x$ and 
            $(g_i \Lambda g^{-1}_i)$ converges to $\{e\}$ in the Chabauty topology.
\end{enumerate}
\end{thm}
If $X$ is minimal, the simplicity of $\cross{X}$ is characterized by purely algebraic conditions for the stabilizer subgroups of $X$, as follows.
\begin{thm}
Let $X$ be a minimal compact $\g$-space. The following are equivalent.
\begin{enumerate}
    \item The reduced crossed product $\cross{X}$ is simple.
    \item For every point $x$ in $X$ and every amenable subgroup $\Lambda$ in $\g_x$,
            there is a sequence $(g_i)$ in $\g$ such that
            $(g_i \Lambda g^{-1}_i)$ converges to $\{e\}$ in the Chabauty topology.
    \item There is a point $x$ in $X$ such that for every amenable subgroup $\Lambda$ in  
            $\g_x$, there is a sequence $(g_i)$ in $\g$ such that
            $(g_i \Lambda g^{-1}_i)$ converges to $\{e\}$ in the Chabauty topology.
\end{enumerate}
\end{thm}
To prove these results, the \emph{equivariant injective envelope} $C(\bX)$ of $C(X)$ plays a central role. The $\g$-space $\bX$ has some properties analogous to those of the \emph{Hamana boundary} (or \emph{universal Furstenberg boundary}, \cite[\S3]{kk}).

The simplicity of reduced crossed products is also characterized in terms of \emph{uniformly recurrent subgroups} (\emph{URS} in short) as with the $\Cst$-simplicity of countable discrete groups \cite{kennedy}.
The notion of URS's is introduced by Glasner--Weiss \cite{gw} as a topological dynamical analogue of the notion of invariant random subgroups, which is an ergodic theoritic concept.
A URS of $\g$ is defined as a minimal component of the $\g$-space $\mathrm{Sub}(\g)$.
The set of all URS's of $\g$ has a natural partial order (denoted by $\preccurlyeq$), introduced by Le Boudec--Matte Bon \cite[\S2.4]{lm}.

The second main result of this paper is a property for amenable URS's from the aspect of its order structure.

\begin{thm} \label{thm2}
Let $X$ be a compact $\g$-space. Suppose that $\cS_X$ is a URS ($X$ is not necessarily  minimal).
Then $\cS_{\bX}$ contains a unique URS $\cA_X$. Moreover, $\cA_X$ is the largest amenable URS dominated by $\cS_X$. Namely, for every amenable URS $\cU$ such that $\cU \preccurlyeq \cS_X$, we have $\cU \preccurlyeq \cA_X$.
\end{thm}
The notation $\cS_X$ denotes the closed $\g$-invariant subspace of $\mathrm{Sub}(\g)$ arising from stabilizer subgroups of $X$, called the \emph{stability system} of $X$ (see \cite[\S1]{gw} or Definition \ref{defstb}). 
If $X$ is minimal, the space $\cS_X$ is a URS. 
On the other hand, every URS is a stability system of a transitive $\g$-space,
but it is not known whether every URS is a stability system of a minimal $\g$-space. 
Using the above result, we prove that it is true for amenable URS's.  

In this paper, we also study the ideals in the group $\Cst$-algebra of $\g$.
In particular, we see the relationship between amenable URS's of $\g$ and the ideals of $\Cstr \g$.  
For an amenable subgroup $\Lambda$ of $\g$, we have the continuous $*$-representation $\pi_{\Lambda}$ of $\Cstr \g$ on the Hilbert space $\ell_2(\g/\Lambda)$ extending the canonical action of $\g$ on the coset space $\g/\Lambda$. 
We show that for stabilizer subgroup $\Lambda$ of the Hamana boundary, the ideal $\mathrm{ker}(\pi_{\Lambda})$ is maximal.

In Section \ref{pre} we recall the notion of stabilizer subgroups and study its relationship to the intersection property.
In Section \ref{inj} we recall the $\g$-injective envelope and show some  properties from the viewpoint of operator algebras which are analogous to those of the Hamana boundary.
In Section \ref{morph} we prove a technical result to prove the main result Theorem \ref{thm1} and we prove it in Section \ref{main}.
In Section \ref{min} we establish the characterization of simplicity of reduced products.
In Section \ref{str} we show a property for the $\g$-injective envelope from the viewpoint of topological dynamical system to prove the main result Theorem \ref{thm2}.
Finally, in Section \ref{maxid} and \ref{ursid} we study the ideals arising from amenable URS's.
\subsection*{Acknowledgements}
The author would like to thank his supervisor, Professor Narutaka Ozawa for his support and many valuable comments. 
\section{Preliminaries} \label{pre}
\begin{defn}
For a compact $\g$-space $X$ and a point $x$ in $X$,
we denote by ${\g}_x$ the \emph{stabilizer subgroup}, 
i.e.\ $ {\g}_x = \{ t \in \g : tx = x \} $.
Let ${\g}_x^\circ$ denote the subgroup consisting the elements in $\g $ 
which act as identity on a neighborhood of $x$.
We say that a compact $\g$-space is \emph{topologically free} if $\g_x^{\circ} = \{ e\}$ for every $x \in X$.
Note that a $\g$-space $X$ is topologically free if and only if $\mathrm{Fix}(t)^{\circ} = \emptyset$ for every $t \in \g \setminus \{e\}$, where $\mathrm{Fix}(t)$ denotes the fixed point set in $X$ of the homeomorphism $t$.
\end{defn}

Let $X$ be a compact $\g$-space. There is a canonical conditional expectation 
$E_X$ from $\cross{X}$ to $C(X)$ defined by 
\[
E_X(f  {\lambda}_t) = 
\begin{cases}
f & t = e  \\
0 & t  \not= e
\end{cases}
\]
and extended by linearity.
Note that $E_X$ is faithful (see \cite[Chapter4.1]{bo}).
For every $x$ in $X$, we define a conditional expectation $E_x$ from $\cross{X}$ to ${\Cstr} ({\g}_x )$ by 
\[
E_x(f {\lambda}_t) = f(x)E_{{\g}_x}({\lambda}_t)
\]
where $E_{{\g}_x}$ is the canonical conditional expectation from
${\Cstr} \g$ to ${\Cstr} ({\g}_x)$ (given by $E_{\g_x} (\lambda_t) = \lambda_t$ if  $t \in \g_x$ and $E_{\g_x} (\lambda_t) = 0$ if $t \in \g \setminus \g_x$, see \cite[Corollary 2.5.12]{bo}).

In this paper, we often use the following fact about unital completely positive maps. See \cite[Proposition 1.5.7]{bo} for proof.
\begin{defn}
Let $A$ and $B$ be unital $\Cst$-algebras and  $\phi$ be a unital completely positive map. The \emph{multiplicative domain} of $\phi$ is the subspace $\mathrm{mult}(\phi)$ of $A$ defined by
\[
\mathrm{mult} (\phi) 
= \{ a \in A : \phi(a^{*} a) = \phi (a)^{*} \phi (a) \ and \ \phi(a a^{*}) = \phi (a) \phi (a)^{*} \}.
\]
\end{defn}
\begin{prop}
Let $A$ and $B$ be unital $\Cst$-algebras and  $\phi$ be a unital completely positive map.
Then, for every $a \in \mathrm{mult}(\phi)$ and $b \in A$, one has $\phi(ab) = \phi(a) \phi(b)$ and $\phi(ba) = \phi(b) \phi(a)$.
In particular, $\mathrm{mult}(\phi)$ is the largest $\Cst$-subalgebra of $A$ to which the restriction of $\phi$ is multiplicative.
\end{prop}
The following is a generalization of the result Theorem \ref{thmkt} and \cite[Theorem 14]{ozawa}.
\begin{lem}\label{lemsta}
Let $X$ be a compact $\g$-space. Then we have the following.
\begin{enumerate}
    \item If the set $\{ x \in X : {\g}_x \ is \ \Cst \mathchar`- simple \}$ is dense in X, then $X$ has the intersection property. 
             In particular, if X is topologically free, then $X$ has the intersection property. 
    \item If  $X$ has the intersection property and ${\g}_x^\circ$ is amenable for every point         
            $x$ in $X$, then $X$ is topologically free. 
\end{enumerate}
\end{lem}
\begin{proof}
We prove (i) by contradiction.
Suppose that there is a non-zero closed ideal $I$ in $\cross{X}$ such that $I \cap C(X) = 0$.
Then $E_X(I)$ is a non-zero since $E_X$ is faithful. 
Therefore $\mathrm{ev}_x \circ E_X (I) \not= 0$ for some $x$ in $X$ such that ${\g}_x$ is $\Cst$-simple (otherwise, we have $\mathrm{ev}_x (E_X(I)) = 0$ densely, this implies that $E_X(I)=0$). 
It follows that $E_x(I) \not= 0$ since $\mathrm{ev}_x \circ E_X = {\tau}_{\lambda} \circ E_x$,
where ${\tau}_{\lambda}$ is the canonical tracial state on $\Cstr ({\g}_x)$ defined by $\tau_\lambda (a) = \langle a \delta_e , \delta_e \rangle$ for any $a \in \Cstr \g$. 
We observe that $E_x(I) \subset \Cstr(\g_x)$ is an ideal of $C(X)$ since $\Cstr(\g_x)$ is contained in the multiplicative domain of $E_x$. 
We show that $E_x(I)$ is not dense in $\Cstr ({\g}_x)$, which yields the desired 
contradiction with $\Cst$-simplicity of $\Cstr ({\g}_x)$. 
The $\ast$-homomorphism
\[
C(X) + I \to (C(X) + I) / I \cong C(X) /(C(X)\cap I) = C(X) \xrightarrow{\mathrm{ev}_x} \IC
\]
extends a state ${\phi}_x$ on $\cross{X}$. 
We show that ${\phi}_x \circ E_x = {\phi}_x$. 
This implies that ${\rm ker} \ {\phi}_x \supset E_x(I)$ since $\phi_x(I)=0$ (hence $E_x(I)$ is not dense). Let $t$ be an element of $\g \setminus \g_x$. There is a function $f \in C(X)$ such that $f(x) = 1$ and $f(tx) = 0$.
Since $C(X)$ is contained in the multiplicative domain of ${\phi}_x$,
we have 
\begin{eqnarray*}
\phi_x (\lambda_t) 
&=& f(x) \phi_x (\lambda_t) 
=\phi_x (f) \phi_x (\lambda_t)
=\phi_x(f\lambda_t) \\
&=&\phi_x(\lambda_t (t^{-1}f))
=\phi_x(\lambda_t)\phi_x(t^{-1}f)
=\phi_x(\lambda_t)f(tx)
=0,
\end{eqnarray*}
therefore $\phi_x = \phi_x \circ E_x$ on $\Cstr \g$.
This implies that for every $f \in C(X)$ and $t \in \lambda$, we obtain 
\[
\phi_x \circ E_x (f\lambda_t) 
= \phi_x (f(x) E_{\g_x}(\lambda_t)) 
= f(x) \phi_x ( E_{\g_x}(\lambda_t)) 
= \phi_x (f) \phi_x (\lambda_t)
= \phi_x (f \lambda_t),
\]
thus we have ${\phi}_x \circ E_x = {\phi}_x$ . 

Next, we show (ii). 
Since ${\g}_x^\circ$ is amenable for any $x$, we define the representation ${\pi}_x$ of $\cross{X}$ on ${\ell}_2 (\g / {\g}_x^\circ)$, which given by
${\pi}_x (f {\lambda}_t) {\delta}_p = f(tpx) {\delta}_{tp}$ for $p\in \g / {\g}_x^\circ$.
Note that for $t$ and $s$ in $\g$ such that $s^{-1}t \in \g_x^{\circ}$, we have $tx=sx$, thus the notation $px$ is well-defined.
Set $\pi = {\bigoplus}_{x \in X} {\pi}_x$. 
The representation $\pi$ is faithful by the intersection property since $\mathrm{ker} (\pi) \cap C(X) = 0$.
This implies that $X$ is topologically free. 
Otherwise, there is an element $t$ in $\g \setminus \{ e\} $ and a non-zero function $f$ in $C(X)$ such that
$\mathrm{supp}(f)$ is contained in $\mathrm{Fix}(t)^{\circ}$, 
which implies that $\pi (f (1 - {\lambda}_t)) =0$ in contradiction with faithfulness.
\end{proof}
\section{Equivariant injective envelopes}\label{inj}
\begin{defn}
We say that an operator system (resp.\ unital $\Cst$-algebra) $V$ is a \emph{$\g$-operator system} (resp.\ \emph{unital $\g$-$\Cst$-algebra})
if it comes together with a complete order isomorphic (resp.\ unital $\ast$-isomorphic) $\g$-action on $V$.
A $\g$-equivariant unital complete positive map between $\g$-operator systems is called a \emph{$\g$-morphism}.
\end{defn} 
\begin{defn}
We say that $\g$-operator system $V$ is \emph{$\g$-injective} if $V$ is an injective object in the category of all $\g$-operator systems with $\g$-morphisms.  
Namely, for any $\g$-operator systems $W_0 \subset W$ and any $\g$-morphism $\phi$ from $W_0$ to $V$, there is a $\g$-morphism $\tilde{\phi}$ from $W$ to $V$ such that $\tilde{\phi}|_{W_0} = \phi$.
\end{defn}
For every compact $\g$-space $X$, we denote by $\bX$ the Gelfand spectrum of the \emph{$\g$-injective envelope} of $C(X)$, i.e.\ $C(\bX)$ satisfies the following properties (see \cite{hamana}).
\begin{itemize}
     \item The $\g$-$\Cst$-algebra $C(\tilde{X})$ is a $\g$-injective operator system. 
     \item The $\g$-$\Cst$-algebra $C(X)$ is contained in $C(\bX)$ as a unital $\g$-$\Cst$-subalgebra and $C(X) \subset C(\bX)$ is rigid, i.e.\ 
the identity map is the only $\g$-morphisms on $C(\bX)$ which is the identity map on $C(X)$.  
\end{itemize}
If $X$ is the one-point $\g$-space, $\bX$ is called \emph{the Hamana boundary}, denoted by ${\partial}_H \g$.

We prove some facts for $\bX$, a generalization of the properties for the Hamana boundary (\cite[Proposition 8 and Lemma 9]{ozawa}).
Recall that a subgroup $\Lambda \le \g$ is \emph{relatively amenable} if there is a $\Lambda$-invariant state on $\ell_{\infty}\g$. Since there is a $\Lambda$-morphism from $\ell_{\infty} \Lambda$ to $\ell_{\infty}\g$, the notions of amenability and relative amenability coincide for discrete groups. 
We denote by $q$ the $\g$-equivariant continuous surjection $\bX$ to $X$.
\begin{prop}\label{propbdy} 
Let $X$ be a compact $\g$-space. Then, one has the following.
\begin{enumerate}
     \item The space $\bX$ is a Stonean space.
     \item For any closed $\g$-invariant set $Z$ in $\bX$, we have $Z = \bX$ if $q(Z) = X$.
     \item The group ${\g}_y$ is amenable for every point $y$ in $\bX$.
\end{enumerate}
In particular, for any $t \in \g$, the set $\mathrm{Fix}(t)$ is clopen, hence $\g_y = \g_y^{\circ}$ for any $y \in \bX$.
\end{prop} 
\begin{proof}
There is an including $\g$-equivariant unital $*$-homomorphism from $C(X)$ to the $\g$-injective $\Cst$-algebra ${\ell}_{\infty}(\g, {\ell}_{\infty}X)$, which is defined by $f \mapsto (tf)_{t \in \g}$. 
It follows that there are $\g$-morphisms $\phi \colon {\ell}_{\infty}(\g, {\ell}_{\infty}X) \to C(\bX)$ and $\psi \colon C(\bX) \to {\ell}_{\infty}(\g, {\ell}_{\infty}X)$, which extend the identity map on $C(X)$.
Since ${\ell}_{\infty}(\g, {\ell}_{\infty}X)$ is also an injective operator system, $C(\bX)$ is an injective operator system, thus $\bX$ is Stonean. 
Then $\mathrm{Fix(t)}$ is clopen by Frol\'ik's theorem.
 
Next we show the condition (ii).
Suppose that there is a closed $\g$-invariant set $Z \subsetneq \bX$ suth that $q(Z) = X$, then the corresponding $\g$-equivariant quotient map $\pi$ from $C(\bX)$ to $C(Z)$ is not faithful.
Since $q(Z) = X$, there is a $\g$-morphism $\phi$ from $C(Z)$ to $C(\bX)$ such that $\phi \circ \pi|_{C(X)} = \id_{C(X)}$ by $\g$-injectivity of $C(\bX)$. This implies that $\phi \circ \pi = \id_{C(\bX)}$ by rigidity, hence $\pi$ is faithful, a contradiction.

Next, we prove amenability of ${\g}_y$. 
There is a inclusion $\iota$ from ${\ell}_{\infty}\g$ to ${\ell}_{\infty}(\g, {\ell}_{\infty}X)$ as a unital $\g$-$\Cst$-subalgebra. 
Since the map $\mathrm{ev}_x \circ \phi \circ \iota$ is a ${\g}_y$-invariant state on $\ell_{\infty}\g$, we obtain (relative) amenability of $\g_y$.
\end{proof}
 We obtain the following result the case $X$ being trivial (see \cite[Theorem 6.2]{kk}).
\begin{thm}\label{thmbdy}
Let $X$ be a compact $\g$-space. Then the following are equivalent.
\begin{enumerate}
    \item The space $X$ has the intersection property.
    \item The space $\bX$ has the intersection property.
    \item The space $\bX$ is (topologically) free.
\end{enumerate} 
\end{thm}
\begin{proof}
First, we prove that (i) implies (ii). Suppose $X$ has the intersection property. 
We show that every quotient map $\pi$ from $\cross{\bX}$ to a $\Cst$-algebra $A$ is faithful if ${\rm ker}(\pi) \cap C(\bX) = 0$.   
Since $\mathrm{ker} (\pi) \cap C(X) = 0$, the quotient map $\pi$ is faithful on $\cross{X}$ by the intersection property for $X$.
By $\g$-injectivity of  $C(\bX)$, there is a $\g$-morphism $\phi$ from $A$ to $C(\bX)$ 
such that  $\phi \circ \pi |_{\cross{X}} = E_X$. 
This implies that $\phi \circ \pi |_{C(\bX)} = \id_{C(\bX)}$ by rigidity of $C(X) \subset C(\bX)$.
Therefore, we obtain $C(\bX) \subset {\rm mult}(\phi \circ \pi)$.
It follows that $\phi \circ \pi  = E_{\bX}$, hence $\pi$ is faithful. 

Next, we prove that (ii) implies (i). Suppose $\bX$ has the intersection property.
Let $\pi$ be a representation of $\cross{X}$ on a Hilbert space $\cH$ such that ${\rm ker} \ \pi \cap C(X) = 0$.
We prove that $\pi$ is injective.
By Arveson's extension theorem, we extend $\pi$ to a unital completely positive map $\tilde{\pi}$ from $\cross{\bX}$ to $\IB (\cH)$. 
We consider a $\Cst$-subalgebra of $\IB (\cH)$ defined by
\[
D = \Cst(\tilde{\pi}(\cross{\bX})) = \mathrm{closure} (\tilde{\pi} (C(\bX)) \cdot \pi(\Cstr \g) ).
\] 
We define $\g$-action on $D$ as $\mathrm{Ad} \ \pi ( \cdot )$, then $\tilde{\pi}$ is $\g$-equivariant.  
Since $\pi$ is faithful on $C(X)$, there is a $\g$-morphism $\phi$ from $\Cst (\tilde{\pi}(C(\bX)))$ to $C(\bX)$ 
such that $\phi \circ \pi = \id_{C(X)}$ by $\g$-injectivity of $C(\bX)$, which implies that $\phi \circ \tilde{\pi}|_{C(\bX)} = \id_{C(\bX)}$ by rigidity.
It follows that $\Cst ( \tilde{\pi}(C(\bX)) ) \subset {\rm mult}(\phi)$, hence $\phi$ is a $*$-homomorphism.
Now, consider a subset of $D$ given by
\[
L = \mathrm{closure} (\mathrm{ker}(\phi) \cdot \pi(\Cstr \g) ).
\]  
Since $\phi$ is $\g$-equivariant, the set $\mathrm{ker}(\phi)$ is $\g$-invariant,
therefore we have 
\[
\mathrm{ker}(\phi)\cdot \pi(\IC(\g)) = \pi(\IC(\g)) \cdot \mathrm{ker} (\phi).
\]
This implies that $L$ is an ideal of $D$.
Since $L \cap \Cst ( \tilde{\pi}(C(\bX)) ) = \mathrm{ker}(\phi)$, 
the map $\phi$ extends the quotient map $\tilde{\phi}$ from $D$ to $D/L$. 
It follows that $\tilde{\phi} \circ \tilde{\pi}$ is a $\ast$-homomorphism which is faithful on $C(\bX)$. 
Thus, we have $\mathrm{ker} (\pi) \subset \mathrm{ker} (\tilde{\phi} \circ \tilde{\pi}) = 0$ 
since $\bX$ has the intersection property.  

The equivalence of (ii) and (iii) follows from Lemma \ref{lemsta} and Proposition \ref{propbdy}. 
\end{proof}
\section{$\g$-morphisms to injective envelopes} \label{morph}
In this section, we prove equivalence of the intersection property and the ``unique trace property'' for crossed products. First, we show a lemma to prove the theorem.
\begin{lem}\label{lemcond}
Let $Y$ be compact $\g$-space. 
If $Y$ is topologically free, then the only conditional expectation from $\cross{Y}$ to $C(Y)$ is the canonical conditional expectation $E_Y$. 
Moreover, if  $Y$ is Stonean and ${\g}_y$ is amenable for every $y$ in $Y$, then the
converse is also true.  
\end{lem}
\begin{proof}
Suppose that $Y$ is topologically free and let $\Phi$ be a conditional expectation from $\cross{Y}$ to $C(Y)$. 
The space $Y$ is topologically free if and only if  $\{ y \in Y :  {\g}_y =\{e\} \}$ (denoted by $Y_0$) is dense in $Y$ since $\bigcup_{t\in \g} \partial \mathrm{Fix}(t)$ has no interior by Baire category theorem.
Fix an element $t$ in $\g \setminus \{e\}$. 
For every $y$ in $Y_0$, we have  $ty \not= y$.
Then there is a non-zero function in $C(X)$ such that $f(y)=1$ and $f {\lambda}_t f = 0$.
It follows that $\Phi ({\lambda}_t)(y) = \Phi (f {\lambda}_t f) (y) = 0$, hence $\Phi (t) =0$. 
This implies that $\Phi = E_Y$.

Next, we show the converse. Suppose that $Y$ is Stonean space and ${\g}_y$ is amenable for every $y$ in $Y$.
There is a conditional expectation $\Phi$ from $\cross{Y}$ to $C(Y)$, defined by
$\Phi (f {\lambda}_t) = f \cdot \chi_{ \mathrm{Fix} (t)}$.
Continuity of $\Phi$ follows from the equality $\Phi( \cdot) (y) = {\tau}_0 \circ E_y$, where ${\tau}_0$ is the unit character of ${\g}_y$. 
Note that ${\tau}_0$ is continuous on $\Cstr (\g_y)$ since $\g_y$ is amenable.
It follows that there is a non-canonical conditional expectation if $Y$ is not (topologically) free. 
\end{proof}
\begin{thm}\label{thmcond} 
Let $X$ be a compact $\g$-space.
Then the following are equivalent. 
\begin{enumerate}
    \item The space $X$ has the intersection property.
    \item The only $\g$-morphism from $\cross{X}$ to $C(\bX)$ which is the identity map on $C(X)$ is the canonical conditional expectation $E_X$.
\end{enumerate}
\end{thm}
\begin{proof}
Let $\phi$ be a $\g$-morphism from $\cross{X}$ to $C(\bX)$ such that $\phi|_{C(X)}=\id_{C(X)}$. 
There is a $\g$-morphism $\Phi$ from $\cross{\bX}$ onto $C(\bX)$ extending $\phi$.
Then $\Phi$ is a conditional expectation since $C(X) \subset C(\bX)$ is rigid.
Hence (ii) is equivalent to the uniqueness of conditional expectations from $\cross{\bX}$ onto $C(\bX)$, that is equivalent to  the (topological) freeness of $\bX$ by Proposition \ref{propbdy} and Lemma \ref{lemcond}. It follows that (i) and (ii) are equivalent by Theorem \ref{thmbdy}.
\end{proof}
\section{Stabilizer subgroups and the intersection property}\label{main}
In this section, we establish a characterization of  the intersection property in terms of stabilizer subgroups.
\begin{defn} \label{defcha}
\emph{The Chabauty space of $\g$} is the set $\mathrm{Sub}(\g)$ of all subgroups in $\g$ with the relative topology of the product topology on $\{0, 1\}^\g $. 
\end{defn}
Note that a sequence $(\Lambda_i)_i$ of subgroup in $\g$ converges to a subgroup $\Lambda$ in the Chabauty topology if and only if it satisfies the following conditions.
\begin{enumerate}
     \item For every $t \in \Lambda$, one has $t \in \Lambda_i$ eventually. 
     \item For every subsequence $(\Lambda_{i_k})_k$ of $(\Lambda_i)_i$, one has $\bigcap_k  \Lambda_{i_k} \subset \Lambda$.  
\end{enumerate}

Let $X$ be a compact $\g$-space. We set the compact $\g$-space
\[
\cS(X, \g) = \{(x , \Lambda) \in X \times \mathrm{Sub}(\g) : \Lambda \leq \g_x \}
\]
with the relative topology of the product topology on $X \times \mathrm{Sub}(\g)$.
We consider the closed $\g$-invariant subspace of $\cS(X, \g)$, defined by
\[
{\cS}_a(X, \g) = \{(x , \Lambda) \in X \times \mathrm{Sub}(\g) : \Lambda \leq \g_x,\ \Lambda \ is \ amenable \}.
\]
We denote by $p_X$ the $\g$-equivariant continuous surjection from ${\cS}_a(X, \g)$ to $X$ defined by 
\[
p_X (x, \Lambda) = x
\]
for $(x, \Lambda) \in \cS_a(X, \g)$, hence $C(X) \subset C(\cS_a(X, \g))$ as a unital $\g$-$\Cst$-subalgebra. 
\begin{thm}\label{thmmain}
Let $X$ be a compact $\g$-space. Then the following are equivalent.
\begin{enumerate}
     \item The space $X$ has the intersection property. 
     \item For every closed $\g$-invariant set $Y$ in ${\cS}_a(X, \g)$ 
             such that $p_X(Y) = X$, the space $Y$ contains $X \times \{ e \}$.  
\end{enumerate}
\end{thm}
\begin{proof}
Suppose that $X$ does not have the intersection property.  
We denote by $q$ the $\g$-equivariant continuous surjection from $\bX$ to $X$.
We define a $\g$-equivariant continuous map $\Phi$ from $\bX$ to ${\cS}_a(X, \g)$ by $\Phi(y) = (q(y), \g_y)$ for $y \in \bX$. 
We claim that $\Phi (\bX) \not \supset X \times \{e\}$, which means that (ii) is not true.
Otherwise, the closed $\g$-invariant set $Z :=\{y \in \bX : \g_y = \{e\} \}$ satisfies that $q(Z) =\bX$, therefore we have $Z = X$ by Proposition \ref{propbdy}. Since $X$ does not have the intersection property, the space $\bX$ is not free by Theorem \ref{thmbdy}, a contradiction.

On the other hand, let $Y$ be a closed $\g$-invariant set in ${\cS}_a(X, \g)$ 
such that $p_X(Y) = X$ and $Y \not\supset X \times \{ e \}$. 
There is a $\g$-morphism $\theta$ from $\cross{X}$ to $C(Y)$, defined by
\[
\theta (f \lambda_t) (x, \Lambda)=
\begin{cases}
f(x) & t \in \Lambda  \\
0    & t \not\in \Lambda
\end{cases}
.\]
There is a $\g$-morphism $\mu$ from $C(Y)$ to $C(\bX)$ which is the identity map on $C(X)$ 
since $C(\bX)$ is $\g$-injective. 
We show that $\mu \circ \theta \not= E_X$.
Suppose that $\mu \circ \theta = E_X$.
We claim that for every $x \in X$, one has $p_X^{-1}(x) \cap Y = \{x \} \times Y_x$ for a $Y_x \subset \mathrm{Sub}(\g)$. 
Since $Y \not\supset X \times \{ e\}$, there is a point $x$ in $X$ 
such that $Y_x \not\ni \{e\}$. 
Let $\tilde{x}$ be a point in $\bX$ such that $q(\tilde{x}) = x$.
We observe that the support of $\mathrm{ev}_{\tilde{x}} \circ \mu$ is contained in $\{x\} \times Y_x$. 
Indeed, let $f$ be a function in $C(Y)$ such that $f|_{\{x\} \times Y_x}=0$. 
For every open neighborhood $U$ of $x$, we take a continuous function $h_U$ on $X$ such that 
$0 \le h_U \le 1$,
the support of $h_U$ contained in $U$ 
and $h_U(x) =1$. 
We denote by $\tilde{h}_U$ the function $h_U \circ q$.
For every $\varepsilon > 0$, there is a neighborhood $U$ of $x$ such that $|f (y)| < \varepsilon$ for every $y \in p_X^{-1} (U) \cap Y$, hence we obtain
\[
  |\mathrm{ev}_{\tilde{x}} \circ \mu (f)| 
=|\mathrm{ev}_{\tilde{x}} \circ \mu  (f) \cdot \tilde{h}_U (\tilde{x})|
=|\mathrm{ev}_{\tilde{x}} \circ \mu (f\tilde{h}_U)| \le \| f\tilde{h}_U\| \le \varepsilon. 
\] 
Therefore  $\mathrm{ev}_{\tilde{x}} \circ \mu (f) = 0$, which implies that $\mathrm{supp}( \mathrm{ev}_{\tilde{x}} \circ \mu) \subset \{x\} \times Y_x$. 
We denote by $\mu_x$ the Radon probability measure on $Y_x$ 
which is the restriction of $\mathrm{ev}_{\tilde{x}} \circ \mu$ on $Y_x$.
Then, for any $t \in \g$, we obtain the following equation.
\begin{eqnarray}
\mu_x \{ \Lambda \in Y_x : t \in \Lambda\} \nonumber
&=& \mu_x (\theta(\lambda_t)|_{Y_x} ) \nonumber \\
&=& \mathrm{ev}_{\tilde{x}} \circ \mu \circ \theta (\lambda_t) \nonumber \\
&=& 
\begin{cases}
1 & (t = e) \\
0 & (t \not= e)
\end{cases} .
\nonumber
\end{eqnarray}
It contradicts that $Y_x \not\ni \{e\}$.
\end{proof}

\begin{thm}
Let $X$ be a compact $\g$-space. The following are equivalent.
\begin{enumerate}
    \item Every $\g$-invariant closed set in $X$ has the intersection property.
    \item For every point $x$ in $X$ and every amenable subgroup $\Lambda$ in $\g_x$,
            there is a net $(g_i)$ in $\g$ such that $(g_i x)$ converges to $x$ and 
            $(g_i \Lambda g^{-1}_i)$ converges to $\{e\}$ in the Chabauty topology.
\end{enumerate}
\end{thm}

\begin{proof}
Suppose that (ii) is true.
Let $Z$ be a $\g$-invariant closed subset of  $X$. 
Then, for any $(z, \Lambda) \in \cS(Z, \g)$, we have $\overline{\g (z, \Lambda)} \ni (z, \{e\})$. 
This implies that for every $\g$-invariant closed subset $Y$ in $\cS_a(Z, \g)$, one has $Z \times \{e\} \subset Y$, therefore $Z$ has the intersection property by Theorem \ref{thmmain}.

Conversly, suppose that (i) is true.
Let $x$ be a point in $X$ and $\Lambda$ be an amenable subgroup in $\g_x$.
Then, we have $\overline{\g (x, \Lambda)} \supset p_X \left(\overline{\g (x, \Lambda)} \right) \times \{ e\}$ by Theorem \ref{thmmain}. 
In particular, there is a net $(g_i)$ in such that $(g_i x)$ converges to $x$ and $(g_i \Lambda g^{-1}_i)$ converges to $\{e\}$.
\end{proof}
\section{Minimal case}\label{min}
We consider the case where the compact $\g$-space $X$ is \emph{minimal}, i.e.\ there are no non-trivial closed $\g$-invariant subspaces in $X$. 
Equivalently, there are no non-trivial $\g$-invariant closed ideals in $C(X)$.
For a minimal compact $\g$-space $X$, the space ${\bX}$ is also minimal by Proposition \ref{propbdy} (ii).

We claim that for a minimal compact $\g$-space $X$, the reduced crossed product $\cross{X}$ is simple if and only if $X$ has the intersection property. 
since for every ideal $I$ in $\cross{X}$, $C(X) \cap I$ is $\g$-invariant ideal.
 
\begin{thm}\label{corsim}
Let $X$ be a minimal compact $\g$-space. The following are equivalent.
\begin{enumerate}
    \item The reduced crossed product $\cross{X}$ is simple.
    \item For every point $x$ in $X$ and every amenable subgroup $\Lambda$ in $\g_x$,
            there is a sequence $(g_i)$ in $\g$ such that
            $(g_i \Lambda g^{-1}_i)$ converges to $\{e\}$ in the Chabauty topology.
    \item There is a point $x$ in $X$ such that for every amenable subgroup $\Lambda$ in  
            $\g_x$, there is a sequence $(g_i)$ in $\g$ such that
            $(g_i \Lambda g^{-1}_i)$ converges to $\{e\}$ in the Chabauty topology.
\end{enumerate}
\end{thm}

\begin{proof}
We show that  (ii) and (iii) are equivalent to existence of a minimal $\g$-invariant subspace $Y$ in $\cS_a(X, \g)$ such that $Y \not = X \times \{e \}$. This implies the desired equivalence by Theorem \ref{thmmain}. Note that  for every $\g$-invariant subspace $Z$ of $X$, we have $p_X(Z) = X$ since $X$ is minimal.

Suppose that there is a minimal $\g$-invariant subspace $Y$ in $\cS_a(X, \g)$ such that $X \times \{e\} \not = Y$.
Let $(x ,\Lambda)$ be an element in $Y$. 
We claim that $\overline{\mathrm{Ad}(\g)\Lambda} \not \ni \{ e\}$, 
which means that (i) and (ii) are not true. 
Otherwise, there is a net $(g_i)$ in $\g$ such that $g_i \Lambda g_i^{-1} \to \{e\}$. 
We may assume that $g_i x \to y$ for a point $y \in X$. 
Then we have $g_i (x, \lambda) \to (y, \{e\})$, this implies that $Y \supset X \times \{e\}$.
By minimality of $Y$, we obtain $Y = X \times \{e\}$, a contradiction.

Next, we show the converse. Suppose that there is an elemant $(x, \Lambda)$ in $\cS_a(X,\g)$ such that $\overline{\mathrm{Ad}(\g)\Lambda} \not \ni \{ e\}$. Then we have $\overline{\g(x, \Lambda)} \cap X\times \{e\} = \emptyset$, hence there is a  minimal $\g$-invariant subspace $Y$ in $\cS_a(X,\g)$ such that $Y \not = X \times \{e\}$ (take a minimal component of $\overline{\g(x, \Lambda)}$).
\end{proof}

We also characterize simplicity for reduced crossed products in terms uniformly recurrent subgroups (Glasner--Weiss \cite{gw}). 

\begin{defn}
A subset $\cU$ of $\mathrm{Sub}(\g)$ is called a \emph{uniformly recurrent subgroup (URS)} of $\g$ if $\cU$ is a minimal closed subset of the Chabauty space $\mathrm{Sub}(\g)$. 
A URS $\cU$ is \emph{amenable} if any subgroup contained in $U$ is amenable.
\end{defn}

\begin{defn} \label{defstb}
For a compact $\g$-space $X$, we define the subspace $\cS_X$ of $\mathrm{Sub}(\g)$ as the closure of the set $\{\g_x : x \in X, \ \g_x = \g_x^{\circ} \}$. We call $\cS_X$ by \emph{stability system} of $X$.
If $X$ is minimal, the set $\cS_X$ is a URS (\cite[Proposition 1.4]{gw}).
\end{defn}

For a normal subgroup $N$ in $\g$, the singleton $\{N\}$ in $\mathrm{Sub}(\g)$ is a URS. 
By \cite[Theorem 4.1]{kennedy}, $\Cst$-simplicity of $\g$ is equivalent to absence of non-trivial amenable URS's.
There is a non-$\Cst$-simple countable group which has no non-trivial normal amenable subgroup, e.g.\ given by Le Boudec \cite{leboudec}.
This implies that there is a countable group which admits a non-singleton amenable URS. 

We define an partial order of the set of uniformly recurrent subgroups. Let $\cU$ and $\cV$ be URS's, we denote $\cU \preccurlyeq \cV$ if there exist $H \in \cU$ and $K \in \cV$ such that $H \leq K$. Note that $\cU \preccurlyeq \cV$ if and only if
for every $H \in \cU$, there exists $K \in \cV$ such that $H \le K$ (see \cite[Propostiton 2.14]{lm}).

\begin{cor}
Let $X$ be a minimal compact $\g$-space. The following are equivalent.
\begin{enumerate}
    \item The reduced crossed product $\cross{X}$ is simple.
    \item For any non-trivial amenable URS $\cU$, one has $\cU \not \preccurlyeq \cS_X$.
\end{enumerate}
\end{cor}

\begin{proof}
If $\cross{X}$ is not simple, $\bX$ is not topologically free by Theorem \ref{thmbdy}. Hence $\cS_{\bX}$ is a non-trivial amenable URS and $\cS_{\bX} \preccurlyeq \cS_X$. 

We show the converse. Suppose that $\cross{X}$ is simple. 
Let $\cU$ be an amenable URS such that $\cU \preccurlyeq \cS_X$. 
For $\Lambda \in \cU$, there is a point $x \in X$ such that $\Lambda \le \g_x$, this implies that $\overline{\mathrm{Ad}(\g) \Lambda} \ni \{e\}$ by Theorem \ref{corsim}.
Therefore we have $\cU = \{e\}$ by minimality of $\cU$.
\end{proof}
\section{Strongly proximallity and amenable URS's} \label{str}
\begin{defn}
Let $X$ be a compact $\g$-space. $X$ is \emph{strongly proximal} if for every Radon probability measure $\mu$ on $X$, the $\mathrm{weak}^\ast$-closure of $\g \mu$ contains a point mass.
$X$ is called a \emph{$\g$-boundary} if $X$ is minimal and strongly proximal.
\end{defn}
It is known that the Hamana boundary ${\partial}_H \g$ is a $\g$-boundary (Kalantar--Kennedy \cite{kk}).
In this section, we proof an analogous property for every compact $\g$-space $X$.

We denote by $\bX_z$ the inverse image of a point $z \in X$ under the $\g$-equivariant continuous surjection from $\bX$ to $X$. For any compact Hausdorff space $Y$, we denote by $\cM (Y)$\ the set of all Radon probability measures on $Y$.

\begin{thm}\label{thmprox}
Let $X$ be a compact $\g$-space and 
$Z$ be a subset in $X$ such that $\g Z$ is dense in $X$. 
Then for every family $\{ \mu_z \}_{z \in Z}$ such that $\mu_z \in \cM (\bX_z)$, 
the space $\bX$ is contained in the $weak^\ast$-closure of $\{t \mu_z : t \in \g , \ z \in Z \}$ in $\cM (\bX)$. 
\end{thm}

\begin{proof}
We define $\g$-morphism $\phi$ from $C(\bX)$ to ${\ell}_{\infty}(\g \times Z)$ by
\[ 
\phi (f)(t, z) = \langle f , t\mu_z \rangle, \ \ \  t \in \g , \ z \in Z.
\]
Observe that for $f \in C(X)$ one has $\phi(f) = (f(tz))_{t, z}$.
Since $\g Z$ is dense in $X$, $\phi|_{C(X)}$ is the inclusion map from $C(X)$ to ${\ell}_{\infty}(\g \times Z)$ as a unital $\g$-$\Cst$-subalgebra.
Hence by $\g$-injectivity of $C(\bX)$, there is a $\g$-morphism $\psi$ from ${\ell}_{\infty}(\g \times Z)$ to $C(\bX)$ which satisfies that $(\psi \circ \phi)|_{C(X)} = \id_{C(X)}$
(then $\psi \circ \phi = \id_{C(\bX)}$ by rigidity). 
It implies that for any $y \in \bX$, there is a state $\omega$ on ${\ell}_{\infty}(\g \times Z)$ such that $\omega \circ \phi = \mathrm{ev}_y$. Since there is a net $(\xi_i)$ in $\ell_1(\g \times Z)$ such that $\xi_i \to \omega$ in $\mathrm{weak}^\ast$-topology, it implies that
\[
\ \xi_i \circ \phi (f) = \sum_{t, z} \xi_i (t, z)  \langle f , t\mu_z \rangle  \to f(y)
\]
for any $f \in C(\bX)$. 
It means that $\mathrm{ev}_y \in \overline{\mathrm{conv}}\{t\mu_z\}$, 
therefore we have $\mathrm{ev}_y \in \overline{\{t\mu_z\}}$ by Milman's converse since $\mathrm{ev}_y$ is an extreme point of $\cM(\bX)$.
\end{proof}

Next, we consider some applications to properties for amenable URS's, inspired from  Le Boudec--Matte Bon \cite[\S 2.4]{lm}.

\begin{lem}\label{lemprox}
Let $X$ be a compact $\g$-space and $\cU$ be an amenable URS.
Suppose that for every $x \in X$, there is a subgroup $H_x \in \cU$ such that $H_x \leq \g_x$.
Then for every $y \in \bX$, there is a subgroup $K_y \in \cU$ such that $K_y \leq \g_y$.
\end{lem}

\begin{proof}
Since $\g_x$ acts on $\bX_x$ and $\cU$ is amenable, there is a $H_x$-invariant measure $\mu_x$ on $\bX_x$ since there is a $\g_x$-morphism from $C(\bX_x)$ to $\ell_{\infty}\g_x$ by $\g_x$-injectivity of $\ell_{\infty}\g_x$ and there is a $H_x$-invariant state on $\ell_{\infty}\g$ by amenability of $H_x$.
Hence, for any $y \in \bX$, there is a net $(t_i , x_i)$ in $\g \times X$ such that $t_i \mu_{x_i} \to \mathrm{ev}_y$ by Theorem \ref{thmprox}. We may assume that the net $(t_i H_{x_i} t_i^{-1})_i$ converges to $K_y \in \cU$. Then $K_y$ fixes $y$ since $t_i \mu_{x_i}$ is ($t_i H_{x_i} t_i^{-1}$)-invariant.
\end{proof}

\begin{thm}[see also Theorem \ref{thm2}] \label{thmurs}
Let $X$ be a compact $\g$-space. Suppose that $\cS_X$ is a URS ($X$ is not necessarily  minimal).
Then $\cS_{\bX}$ contains a unique URS $\cA_X$. Moreover, for every amenable URS $\cU$ such that $\cU \preccurlyeq \cS_X$, we have $\cU \preccurlyeq \cA_X$.
\end{thm}

\begin{proof}
Let $\cU$ be a URS such that $\cU \preccurlyeq \cS_X$. Since $\cS_X$ is a URS, for any $x \in X$ there is a subgroup $H \in \cU$ such that $H \leq \g_x$. Hence for every $y \in \bX$, there is a subgroup $K_y \in \cU$ such that $K_y \leq \g_y$ by Lemma \ref{lemprox}.  It implies that $\cU \preccurlyeq \cV$ for every URS $\cV$ in $\cS_{\bX}$. 
In particular, for every URS $\cV$ in $\cS_{\bX}$, one has $\cV \preccurlyeq \cS_X$ since $\g_y = \g_y^{\circ}$ for every $y \in \bX$ by Proposition \ref{propbdy}. This implies that $\cV_1 \preccurlyeq \cV_2$ for URS's $\cV_1$ and $\cV_2 \subset \cS_{\bX}$, hence  $\cV_1 = \cV_2$. Therefore, there is a unique URS contained in $\cS_{\bX}$.
\end{proof}

For every URS $\cU$, there is a compact $\g$-space $X$ such that $\cS_X = \cU$ (\cite[Proposition 6.1]{gw}). Hence we get the following.
\begin{cor}
For every URS $\cU$, there is a unique amenable URS $\cA_{\cU} \preccurlyeq \cU$  which satisfies that $\cV \preccurlyeq \cA_{\cU}$ for every amenable URS $\cV \preccurlyeq \cU$.
\end{cor}
It is not known whether for every URS $\cU$, there exists a minimal $\g$-space $X$ such that $\cS_X = \cU$. Here, we prove that it is true for amenable URS's.  
\begin{cor}
For every amenable URS $\cU$, there is a minimal compact $\g$-space $X$ such that $\cS_X = \cU$. 
\end{cor}
\begin{proof}
There is a compact $\g$-space $X$ such that $\cS_X = \cU$. We take a minimal $\g$-subspace $Y$ in $\bX$, then $\cS_Y \subset \cS_{\bX}$ since $\g_y = \g_y^{\circ}$ for every $y \in \bX$. Hence we have $\cU \preccurlyeq \cS_Y = \cA_X \preccurlyeq \cU$ by Theorem \ref{thmurs}. 
\end{proof}
\section{The maximal ideal arising from stabilizer subgroups}\label{maxid}
Let $X$ be a minimal compact $\g$-space (recall that $\bX$ is also minimal in this situation). 
For $x \in \bX$, we define a representation $\pi_x$ of $\cross{\bX}$ on $\ell_2(\g x)$ as follows.
\[
\begin{cases}
\pi_x (f) \delta_y  = f(y)\delta_y & f \in C(\bX) \\
\pi_x (\lambda_t) \delta_y = \delta_{ty} & t \in \g
\end{cases},
\]
where $y \in \g x$.
In other words, $\pi_x$ is the GNS representation with respect to $1_x :=\tau_{0} \circ E_x$, where ${\tau}_0$ is the unit character of ${\g}_x$. Since $\tau_0$ is continuous since $\g_x$ is amenable by Proposition \ref{propbdy}, the state $1_x$ is continuous. 

Note that $\mathrm{ker}(\pi_x) = \mathrm{ker}(\pi_y)$ for every $x, \ y \in \bX$ since $1_{tx} = 1_x \circ \mathrm{Ad}(t^{-1})$ for every $t \in \g$ and the map $\bX \to S(\cross{\bX})$ given by $\ x \mapsto 1_x$ is continuous, where $S(\cross{\bX})$ is the state space of $\cross{\bX}$. 
\begin{thm}\label{thmmax}
For every minimal compact $\g$-space $X$ and every $x \in \bX$, the $\Cst$-algebra $\pi_x ( \cross{X})$ is simple.
\end{thm}
First, we prove a lemma. We define the unital completely positive map $\tilde{E}$ on $\cross{\bX}$ by $\tilde{E}(f \lambda_t) = f \chi_{\mathrm{Fix}(t)} \lambda_t$ for $f \in C(\bX)$ and $t \in \g$.
We see that $\tilde{E}$ is continuous. 
Let $B \subset \cross{\bX}$ the closed linear span of  $\{ f \lambda_t : \mathrm{supp} (f) \subset \mathrm{Fix}(t) \}$.
Then, $B$ is a C*-subalgebra of $C(\bX)\rtimes_r\g$, which is contained in the multiplicative domain of $E_x$ for every $x \in \bX$.
Since $E_x\circ\tilde{E} = E_x$ and $\{E_x\}_x$ is a faithful family of $*$-homomorphisms on $B$ (because $\tau_\lambda\circ E_x = \mathrm{ev}_x \circ E_{\bX}$), the map $\tilde{E}$ is continuous on $C(\bX)\rtimes_r\g$.
Note that $\pi_x \circ \tilde{E} (\lambda_t) = \pi_x (\chi_{\mathrm{Fix}(t)})$ for every $t \in \g$.
\begin{lem}\label{lemmax}
For every conditional expectation $\Phi \colon \cross{\bX} \to C(\bX)$, one has $\Phi \circ \tilde{E} = \Phi$. 
\end{lem}
\begin{proof}
It suffices to show that $\Phi(\lambda_t)(x) = 0$ for every $x \in \bX \setminus \mathrm{Fix}(t)$. Since $x \not \in \mathrm{Fix}(t)$, there is a $f \in C(\bX)$ such that $f(x) =1$ and $f(tx) =0$.
Then $\Phi(\lambda_t) (x) = f(x) \Phi(\lambda_t) (x) = \Phi(f \lambda_t) (x)= \Phi ( \lambda_t(t^{-1}f)) (x) = \Phi (\lambda_t) (x) f(tx) = 0$.
\end{proof}
\begin{proof}[Proof of Theorem \ref{thmmax}]
Suppose that there is a quotient map $\rho$ from $\cross{X}$ to a unital $\Cst$-algebra $A$. It suffices to show that $\rho$ is faithful.
Since $X$ is minimal, the map $\rho \circ \pi_x$ is injective on $C(X)$, hence there is $\g$-morphism $\phi \colon A \to C(\bX)$ such that $\phi \circ \rho \circ \pi_x |_{C(X)} = \id_{C(X)}$ by $\g$-injectivity. 
We extend $\phi \circ \rho$ to a $\g$-morphism $\Phi \colon \pi_x(\cross{\bX} )\to C(\bX)$ such that $\Phi \circ \pi_x|_{C(X)} = \id_{C(X)}$ by $\g$-injectivity of $C(\bX)$. Then $\Phi \circ \pi_x$ is a conditional expectation by rigidity. 
By Lemma \ref{lemmax} and the fact that $\pi_x \circ \tilde{E} (\lambda_t) = \pi_x (\chi_{\mathrm{Fix}(t)})$, for $t \in \g$, we obtain the following equality.
\begin{eqnarray*}
\mathrm{ev}_x \circ \phi \circ \rho \circ \pi_x (\lambda_t)  & = & 
\mathrm{ev}_x \circ \Phi \circ \pi_x (\lambda_t)  \\ & = &
\mathrm{ev}_x \circ \Phi \circ \pi_x \circ \tilde{E} (\lambda_t) \\ & = &
\mathrm{ev}_x \circ \Phi \circ \pi_x (\chi_{\mathrm{Fix}(t)}) \\ & = &
\chi_{\mathrm{Fix}(t)} (x) \\ & = &
\begin{cases}
1 & t \in \g_x \\
0 & t \not\in \g_x
\end{cases} \\ & = &
1_x (\lambda_t).
\end{eqnarray*}
Since $C(X)$ is contained in the multiplicative domains of $1_x $ and $\phi \circ \rho$,
for any $f \in C(X)$ and $t \in \g$, we have
\begin{eqnarray*}
1_x (f \lambda_t) = f(x) 1_x(\lambda_t) &=& 
f(x)\mathrm{ev}_x \circ \phi \circ \rho \circ \pi_x (\lambda_t) \\ &=&
\mathrm{ev}_x (f  \phi \circ \rho ( \pi_x (\lambda_t))) \\ &=&
\mathrm{ev}_x \circ \phi \circ \rho(f \pi_x (\lambda_t)) \\ &=&
\mathrm{ev}_x \circ \phi \circ \rho \circ \pi_x (f \lambda_t).
\end{eqnarray*}
This implies that $1_x = \langle \pi_x (\cdot) \delta_x , \delta_x \rangle = \mathrm{ev}_x \circ \phi \circ \rho \circ \pi_x$ on $\cross{X}$. 
Since $\delta_x$ is a cyclic vector, $\rho$ is faithful.
\end{proof}

In particular, for any $x \in \partial_H \g$ (recall that $C(\partial_H \g)$ is the $\g$-injective envelope of $\IC$), the $\Cst$-algebra $\pi_x (\Cstr \g)$ is simple. Moreover, we have a stronger conclusion in this situation, a generalization of the Powers' averaging property (\cite{powers}), which is equivalent to the $\Cst$-simplicity (see \cite{haagerup, kennedy}). 
First, we show the following lemma.

\begin{lem}\label{lempap} 
Let $x$ be a point in $\partial_H \g$. 
Then for every finite family $\{ \phi_k \}_{k=0}^{N}$ of states on $\pi_x (\cross{\partial_H \g})$, there is a net $(\alpha_i)$ in $\mathrm{conv}\{ \mathrm{Ad}(t) : t \in \g \}$ such that $\phi_k \circ \pi_x \circ \alpha_i \to 1_x$ for every index $k$.
\end{lem}

\begin{proof}
First we show that for every state $\phi$ on $\pi_x (\cross{\partial_H \g})$,  we have
\[
1_x \in \overline{\mathrm{conv}}\{\phi \circ \pi_x  \circ \mathrm{Ad}(t) :t \in \g \}.
\]
By \cite[Theorem 2.3]{glasner}, there is a $\g$-boundary $X \subset  \overline{\mathrm{conv}}\{\phi \circ \pi_x \circ \mathrm{Ad}(t) :t \in \g \}$.
Hence there is a $\g$-equivariant continuous surjection $p : \partial_H \g \to X$ by \cite[Theorem 3.11]{kk}.
Since there is a natural $\g$-morphism from $\pi_x (\cross{\partial_H \g})$ to $C(X)$, we have a $\g$-morphism $\Phi$ from $\pi_x (\cross{\partial_H \g})$ to $C(\partial_H \g)$ such that $\Phi (\pi_x (a))(x) = p(x) (a)$ for any $a \in \cross{\partial_H \g}$.
Then $\Phi \circ \pi_x$ is conditional expectation from $\cross{\partial_H \g}$ to $C(\partial_H \g)$. 
Hence by Lemma \ref{lemmax}, we have
\[
p(x)(\pi_x(\lambda_t))= \Phi \circ \pi_x (\lambda_t) (x) =  \Phi \circ \pi_x \circ \tilde{E} (\lambda_t) (x)
= \chi_{\mathrm{Fix}(t)} (x) = 1_x (\lambda_t),
\]
hence $1_x \in X$.

Next, we show the theorem. Take a net $(\alpha_i)$ in $\mathrm{conv}\{ \mathrm{Ad}(t) : t \in \g \}$ such that
\[
\frac{1}{N} (\sum\phi_k) \circ \pi_x  \circ \alpha_i \to 1_x.
\]
We may assume that $\phi_k \circ \alpha_i \to \psi_k$, where $\psi_k \in S(\pi_x ( \cross{\partial_H \g}))$, then we have 
\[
\frac{1}{N} (\sum\psi_k)\circ \pi_x = 1_x.
\]
This implies that $\psi_k \circ \pi|_{\Cstr (\g_x)} = 1_x|_{\Cstr (\g_x)} = \tau_0$ because $\tau_0$ is a character, hence it is extremal in $S(\Cstr(\g_x))$. 
Similarly, we obtain $\psi_k \circ \pi|_{C(\partial_H \g)} = 1_x|_{C(\partial_H \g)} = \mathrm{ev}_x$ because 
$\mathrm{ev}_x$ is an extreme point of $\cM(\bX)$.
We claim that for any $\theta \in S(\cross{\partial_H \g})$,  $\theta|_{\Cstr(\g_x)}= \tau_0$ and $\theta|_{C(\partial_H \g)} = \mathrm{ev}_x$ imply that $\theta = 1_x$. Since $C(\partial_H \g) \subset \mathrm{mult}(\theta)$, it suffices to show that $\theta(t) = 0$ for every $t \in \g \setminus \g_x$. Take a function $f \in C(\bX)$ such that $f(x)=1$ and $f(tx) = 0$, we have
$1_x(\lambda_t) = f(x) 1_x(\lambda_t) = 1_x(f \lambda_t) = 1_x (\lambda_t (t^{-1}f) ) = 1_x (\lambda_t) f(tx) = 0$. 
Hence we have $\psi_k \circ \pi_x = 1_x$.
\end{proof}

\begin{thm}\label{thmpap}
Let $x$ be a point in $\partial_H \g$.
Then for every $a \in \Cstr \g$, the element $1_x (a)$ is contained in the norm closed convex hull of $\{ \pi_x (\lambda_t a \lambda_t^{\ast}) : t \in \g \}$.
\end{thm}

\begin{proof}
We show it by contradiction.
Suppose that there is an element $a \in \pi_x (\Cstr \g)$ such that $1_x (a) \not \in \overline{\mathrm{conv}}^{\mathrm{norm}}\{ \pi_x (\lambda_t a \lambda_t^{\ast}) : t \in \g \}$.
Then  there is a bounded linear functional $\phi$ on $\pi_x(\Cstr \g)$ such that $\mathrm{Re}(\phi \circ \pi_x (b) - \phi(1)1_x (a) ) \ge \varepsilon > 0$ for every $b \in \overline{\mathrm{conv}}^{\mathrm{norm}}\{ \pi_x (\lambda_t a \lambda_t^{\ast}) : t \in \g \}$ by Hahn--Banach separation theorem.
By Hahn--Jordan decomposition, we can write $\phi = \sum_{k=0}^3 i^k c_k \phi_k$, where $\phi_k$ is state on  $\pi_x (\Cstr \g)$ and $c_k$ is non-negative scalar. 
Then by Lemma \ref{lempap}, there is a net $(\alpha_i)$ in $\mathrm{conv}\{ \mathrm{Ad}(t) : t \in \g \}$ such that $\phi_k \circ \pi_x \circ \alpha_i \to 1_x|_{\Cstr \g}$ for $k=1, 2, 3, 4$.
But we have 
\[
  \mathrm{Re}(\phi \circ \pi_x \circ \alpha_i (a) - \phi(1)1_x (a) ) = 
  \mathrm{Re}\left( \sum_{k=0}^3 i^k c_k( \phi_k \circ \pi_x \circ \alpha_i (a)  - 1_x (a)) \right) \ge \varepsilon,
\]
a contradiction.
\end{proof}

\begin{cor}\label{corpap}
For every point $x$ in $\partial_H \g$, the $\Cst$-algebra $\pi_x(\Cstr \g)$ is simple.
\end{cor}

\begin{proof}
For every non-zero positive element $\pi_x(a) \in \pi_x(\Cstr \g)$, we have $1_x(a) \not= 0$ since $1_x$ is faithful on $\pi_x(\Cstr \g)$.
This implies that $1_x(a) \in \mathrm{Ideal}(\pi_x(a))$ by theorem \ref{thmpap}, where $\mathrm{Ideal}(\pi_x(a))$ is the ideal in $\pi_x(\Cstr \g)$ generated by $\{\pi_x(a)\}$.
\end{proof}

\section{Amenable URS's and ideals in the group $\Cst$-algebra} \label{ursid}

In this section, we see the relationship between amenable URS's of $\g$ and ideals in $\Cstr \g$.
For an amenable subgroup $\Lambda$ in $\g$, we define a representation $\pi_{\Lambda}$ of $\Cstr \g$ on $\ell_2(\g / \Lambda)$ by 
\[
\pi_\Lambda (\lambda_t) \delta_{x} = \delta_{tx},  \ x \in \g/\Lambda . 
\]
Since $\langle\pi_{\Lambda} ( \cdot )\delta_{\Lambda} ,  \delta_{\Lambda}\rangle = \tau_0 \circ E_{\Lambda}$ (where $\tau_0$ is the unit character), the representation $\pi_{\Lambda}$ is unitarily equivalent to the GNS representation with respect to $1_{\Lambda} := \tau_0 \circ E_{\Lambda}$.
Note that for $x\in \partial_H \g$, one has $\pi_x|_{\Cstr \g} = \pi_{\g_x}$. 
Since $1_\Lambda \circ \mathrm{Ad}(t^{-1}) = 1_{t\Lambda t^{-1} }$, we have the following equality.
\begin{eqnarray*}
\mathrm{ker}(\pi_\Lambda) 
&=& 
\{ a \in \Cstr \g : \langle\pi_{\Lambda} (a)\xi ,  \eta \rangle=0, \ for \ every \ \xi, \eta \in \ell_2 (\g/\Lambda)\}\\
&=&
\{ a \in \Cstr \g : \langle\pi_{\Lambda} ( \cdot )\delta_{t\Lambda} ,  \delta_{s\Lambda}\rangle=1_{\Lambda}(\lambda_{s}^{*} a \lambda_t) =0, \ for \ every \ s, t \in \g \}\\
&=&
\bigcap_{s, t \in \g} \{a \in \Cstr \g : 1_{\Lambda}(\lambda_{s} a \lambda_t) =0\}\\
&=&
\bigcap_{\Delta \in \mathrm{Ad}(\g)\Lambda} \bigcap_{t\in \g} \{a \in \Cstr \g : 1_{\Delta}(at)=0\}\\
&=&
\bigcap_{\Delta \in \overline{\mathrm{Ad}(\g)\Lambda}} \bigcap_{t\in \g} \{a \in \Cstr \g : 1_{\Delta}(at)=0\}.
\end{eqnarray*}
In particular, for every amenable URS $\cU$ and every elements $H_1$ and $H_2$ in $\cU$, we have $\mathrm{ker}(\pi_{H_1}) = \mathrm{ker} (\pi_{H_2})$, hence we set $I_{\cU} := \mathrm{ker}(\pi_H)$ for $H \in \cU$. Note that $I_{\cS_{\partial_H \g}}$ is maximal by Corollary \ref{corpap}.
From the above equality, we obtain the following easily.
\begin{prop}
Let $\Lambda$ be an amenable subgroup of $\g$. Then for every amenable URS $\cU$ contained in $\overline{\mathrm{Ad}(\g)\Lambda}$, we have $\mathrm{ker}(\pi_\Lambda) \subset I_{\cU}$. 
\end{prop}
In particular, if $\{e\} \in \overline{\mathrm{Ad}(\g)\Lambda}$, the representation $\pi_{\Lambda}$ is faithful, but the converse need not be true in general, i.e.\ there is a group which has a non-trivial amenable URS $\cU$ such that $I_{\cU} = 0$.
The following example was communicated to us by Koichi Shimada.
\begin{exam}\label{exid}
Let $A_4$ denote the alternating group on 4 letters $\{1,2,3,4\}$.
Then, the group algebra ${\mathbb{C}}(A_4)$ is isomorphic to
${\mathbb{C}}^3\oplus {\mathbb{M}}_3(\mathbb{C})$.
Indeed, the derived subgroup of $A_4$ is $K:=\{ e, (1,2)(3,4),
(1,3)(2,4), (1,4)(2,3)\}$
and $A_4/K\cong{\mathbb{Z}}/3{\mathbb{Z}}$, which accounts for the abelian quotient
${\mathbb{C}}(A_4/K) \cong {\mathbb{C}}^3$.
Since the standard action of $A_4$ on the 4 letters $\{1,2,3,4\}$ is
doubly transitive,
it gives rise to an irreducible representation on the 3-dimensional space
$\{ \xi \in {\mathbb C}(\{1,2,3,4\}) : \sum_k\xi(k)=0\}$, which accounts for
the factor ${\mathbb M}_3({\mathbb{C}})$.
Since $\dim {\mathbb C}(A_4) = |A_4| = 12 = \dim {\mathbb C}^3\oplus {\mathbb
M}_3({\mathbb C})$,
we have ${\mathbb C}(A_4) \cong {\mathbb C}^3\oplus {\mathbb M}_3({\mathbb C})$.
Now we consider the subgroup $\Lambda := \{ e, (1,2)(3,4)\}$ of order 2.
 From the above description, it is not difficult to see that the representation
$\pi_\Lambda$ of ${\mathbb C}(A_4)$ on ${\mathbb C}(A_4/\Lambda)$ is faithful.
\end{exam}
\begin{prop}\label{propid}
Let $\cU$ and $\cV$ be amenable URS's. If $\cU \preccurlyeq \cV$, one has $I_{\cU} \subset I_{\cV}$.
\end{prop}
\begin{proof}
It suffices to show that for amenable subgroups $\Lambda_1$ and $\Lambda_2$ in $\g$ such that $\Lambda_1 \le \Lambda_2$,  we have $\mathrm{ker}(\pi_{\Lambda_1}) \subset \mathrm{ker}(\pi_{\Lambda_2})$.
By amenability of $\Lambda_2$, there is an approximately invariant vector $(\xi_i)$ in $\ell_2 (\Lambda_2 / \Lambda_1)$, i.e.\ $(\xi_i)$ is a net in $\ell_2 (\Lambda_2 / \Lambda_1) \subset \ell_2 (\g / \Lambda_1)$ such that $\|\pi_{\Lambda_1}(\lambda_t)\xi_i - \xi_i\| \to 0$ for any $t\in \Lambda_2$. 
Take a F{\o}lner net $(F_i)$ of $\Lambda_2$, then the net of vectors 
\[
\xi_i := |F_i|^{-1} \sum_{t \in F_i} \delta_{t\Lambda}
\]
is approximately invariant in $\ell_1(\Lambda_2/\Lambda_1)$, hence the net $(\xi_i^{1/2})$ is an approximately invariant in $\ell_2(\Lambda_2/\Lambda_1)$.
Then we have the net $(\langle \pi_{\Lambda_1}(\cdot)\xi_i, \xi_i \rangle)$ of state on $\Cstr  \g$ converges to $1_{\Lambda_2}$. 
This implies that there is a state $\phi$ on $\pi_{\Lambda_1}(\Cstr \g)$ such that $1_{\Lambda_2} = \phi \circ \pi_{\Lambda_1}$, hence we have $\mathrm{ker}(\pi_{\Lambda_1}) \subset \mathrm{ker}(\pi_{\Lambda_2})$.
\end{proof}
We show a relaxed form of the converse of Proposition \ref{propid}. 
(Note that the converse of Proposition \ref{propid} is not true. 
Example \ref{exid} is a counter example.) 
For a subset $S$ in $\g$, we set $T(S) := \{t \in \g : t^n \in S \ \mbox{for a non-zero integer } n \}$.
\begin{thm} \label{thmid}
Let $\Lambda$ and ${\Lambda}'$ be amenable subgroups of $\g$ such that $\mathrm{ker}(\pi_{\Lambda}) \subset \mathrm{ker}(\pi_{\Lambda'})$. Then, there is an amenable subgroup $\Delta \in \overline{\mathrm{Ad}(\g)\Lambda}$ such that $\Delta \subset T(\Lambda')$.
\end{thm}
\begin{proof}
It suffices to show that for every finite set $F \subset \g \setminus T(\Lambda')$, there is an element $t_F \in \g$ such that $F \subset \g \setminus t_F \Lambda t_F^{-1}$.
Indeed, let $(F_n)$ be an increasing sequence of finite subset in $\g \setminus T(\Lambda')$ such that $\bigcup F_n = \g \setminus T(\Lambda')$. 
Take a cluster point $\Delta$ of $\{ t_{F_n} \Lambda t_{F_n}^{-1}\}$ where $t_{F_n}$ satisfies that $F_n \subset \g \setminus t_{F_n} \Lambda t_{F_{n}}^{-1}$, then we have $\Delta \subset T(\Lambda')$.
We show it by contradiction. 
Suppose that there is a finite set $F \subset \g \setminus T(\Lambda')$ such that $t\Lambda t^{-1} \cap F \not= \emptyset$ for every $t \in \g$.
It is easy to see that $t\Lambda t^{-1} \cap F \not= \emptyset$ for every $t \in \g$ if and only if for all $x \in \g/\Lambda$, there exists an element $g \in F$ such that $g x =x$.
We define $p_g$ as the orthogonal projection onto the closed linear span of $\{\delta_x : x \in \g / \Lambda , \ g x = x\}$.
Then we obtain the following conditions.
\begin{itemize}
\item $\pi_{\Lambda'}(\lambda_g) p_g = p_g = p_g \pi_{\Lambda'}(\lambda_g)$ for every $g \in F$. 
\item $\sum_{g \in F} p_g \ge 1$.
\end{itemize}  
Since $\mathrm{ker}(\pi_{\Lambda}) \subset \mathrm{ker}(\pi_{\Lambda'})$, the map $\pi_{\Lambda} (\Cstr \g) \ni \pi_{\Lambda} (a) \to \pi_{\Lambda'}(a) \in \pi_{\Lambda'}(\Cstr \g)$ is a $*$-homomorphism.
We extend it to a unital completely positive map $\Theta$ from $\IB(\ell_2 (\g /\Lambda))$
to $\IB(\ell_2 (\g /\Lambda'))$ by Arveson's extension theorem. 
Since $\pi_{\Lambda}(\Cstr \g) \subset \mathrm{mult} (\Theta)$, the element $a_g := \Theta (p_g)$ satisfies the following conditions.
\begin{itemize}
\item $0 \le a_g \le 1$ for every $g \in F$. 
\item $\pi_{\Lambda'}(\lambda_g) a_g = a_g = a_g \pi_{\Lambda'}(\lambda_g)$ for every $g \in F$. 
\item $\sum_{g \in F} a_g \ge 1$.
\end{itemize}  
The sequence $n^{-1} \sum_{k=1}^n \pi_{\Lambda'}(\lambda_g)^k$ converges in the strong operator topology to the orthogonal projection onto the $\pi_{\Lambda'}(\lambda_g)$-invariant vectors, which will be denoted by $q_g$.
The second condition implies that $q_g a_g = a_g q_g$, therefore we have $\mathrm{supp}(a_g) \le q_g$. 
Since $g^n \not \in \Lambda'$ for every non-zero integer $n$, we have $\langle  \pi_{\Lambda'} (\lambda_g)^n \delta_{\Lambda'}, \delta_{\Lambda'} \rangle =0$.
This implies that  
\[
\langle \mathrm{supp}(a_g) \delta_{\Lambda'}, \delta_{\Lambda'} \rangle  \le \langle  q_g \delta_{\Lambda'}, \delta_{\Lambda'} \rangle =0.
\]
Hence we have $\langle a_g \delta_{\Lambda'}, \delta_{\Lambda'} \rangle =0$, it contradicts that
\[
\langle \sum_{g \in F}  a_g \delta_{\Lambda'}, \delta_{\Lambda'} \rangle \ge \langle \delta_{\Lambda'}, \delta_{\Lambda'} \rangle =1.
\]
\end{proof}
\begin{cor}
Let $\Lambda$ be an amenable subgroup of $\g$ such that the representation $\pi_{\Lambda}$ is faithful.
Then, there is a torsion group $\Delta$ contained in $\overline{\mathrm{Ad}(\g)\Lambda}$.
In particular, for any amenable URS $\cU$, the condition $I_{\cU} = 0$ implies that $\cU$ consists of torsion groups.
\end{cor}
\begin{proof}
It is easy to see the first part of  the theorem by Theorem \ref{thmid}.
Let $\cU$ be an amenable $URS$ such that $I_{\cU} =0$. Then, there is a torsion group $\Delta \in \cU$. Since $\overline{\mathrm{Ad}(\g)\Delta} = \cU$, every $H \in \cU$ is a torsion group. 
\end{proof}
Note that the converse of the above corollary need not be true in general. 
Let $N$ be a non-trivial finite normal subgroup of $\g$. 
Then, it is clear that $\pi_N$ is not faithful, but $N$ is a torsion group since $|N|$ is finite. 

\end{document}